\documentclass[11pt,twoside]{amsart}

\usepackage[english]{babel}
\usepackage{amsmath, amsthm, amssymb, amsfonts}
\usepackage{latexsym}
\usepackage{graphicx}
\usepackage{pdfpages}
\usepackage {bm}
\usepackage {indentfirst} 

\usepackage{hyperref,cite}

\advance\hoffset by -1.5cm

\newcommand{\De}{\mathbb{D}}
\newcommand{\C}{\mathbb{C}}
\newcommand{\N}{\mathbb{N}}

\newcommand{\Z}{\mathbb{Z}}

\newcommand{\B}{\mathcal{B(H)}}
\newcommand{\D}{\mathcal{D}}
\newcommand{\R}{\mathcal{R}}
\newcommand{\n}{\mathcal{N}}
\newcommand{\h}{\mathcal{H}}
\newcommand{\e}{\mathcal{E}}

\newcommand{\ka}{\mathcal{K}}

\renewcommand{\Re}{\operatorname{Re}}

\newtheorem{theorem}{Theorem}[section]
\newtheorem{lemma}[theorem]{Lemma}
\newtheorem{proposition}[theorem]{Proposition}
\newtheorem{corollary}[theorem]{Corollary}

\theoremstyle{definition}

\newtheorem{example}[theorem]{Example}

\theoremstyle{definition}
\newtheorem{remark}[theorem]{Remark}

\newcommand{\echiH}{\begin{array}[b]{c}
                  \mbox{\tiny{H}}\\[-0.3cm]
                  \sim
                  \end{array}}
        \newcommand{\echiK}{\begin{array}[b]{c}
                  \mbox{\tiny{K}}\\[-0.3cm]
                  \sim
                  \end{array}}

\newcommand{\Prec}{\stackrel{\rm H}{\prec}}

\newcommand{\PrecH}{\stackrel{\rm \infty}{\prec}}
\newcommand{\PrecS}{\stackrel{\rm Sh}{\prec}}

\numberwithin{equation}{section}

\title[Harnack and Shmul'yan pre-order relations]{Harnack and Shmul'yan pre-order relations for Hilbert space contractions}

\author[C. Badea]{Catalin Badea}
\address{Universit\'e de Lille, CNRS, UMR 8524 - Laboratoire Paul Painlev\'e, F-59655 Villeneuve d'Ascq Cedex, France}

\email{catalin.badea@univ-lille.fr}

\author[L. Suciu]{Laurian Suciu}
\address{Department of Mathematics, "Lucian Blaga" University
	of Sibiu, Dr. Ion Ra\c tiu 5-7, Sibiu, 550012, Romania}
\email{laurians2002@yahoo.com}



\begin{document}
\begin{abstract}
	We study the behavior of some classes of Hilbert space contractions with respect to Harnack and Shmul'yan pre-orders and the corresponding equivalence relations. We give some conditions under which the Harnack equivalence of two given contractions is equivalent to their Shmul'yan equivalence and to the existence of an arc joining the two contractions in the class of operator-valued contractive analytic functions on the unit disc. We apply some of these results to quasi-isometries and quasi-normal contractions, as well as to partial isometries for which we show that their Harnack and Shmul'yan parts coincide. We also discuss an extension, recently considered by S.~ter~Horst [\emph{J. Operator Th. 72(2014), 487--520}], of the Shmul'yan pre-order from contractions to the operator-valued Schur class of functions. In particular, the Shmul'yan-ter Horst part of a given partial isometry, viewed as a constant Schur class function, is explicitly determined.
	\end{abstract}
	\maketitle

\section{Introduction}  \label{Section:Intro}
\medskip

{\bf Preamble.} Let $\h$ be a complex Hilbert space and let $\mathcal{B}_1(\h)$ denote the unit ball of the C$^\ast$-algebra $\mathcal{B}(\h)$ of all bounded linear operators on $\h$. Following the usual terminology of operator theory, elements of $\mathcal{B}_1(\h)$ are called contractions. One tool in the study of the (hyperbolic) geometry of $\mathcal{B}_1(\h)$ is the use of order relations such as the Harnack and Shmul'yan pre-orders. Both pre-order relations have nice geometric and analytic interpretations. Although these two pre-orders have been around since 1970s and 1980s \cite{AST,F,Sh,S1,S2,SV,Sn}, their structure is to date not completely understood, and in recent years there has been an increase in interest for this topic \cite{BST,H,H1,KSS,Po1,Po2,S4,S3,SS}.

The aim of this paper is to study the behavior of some classes of Hilbert space contractions with respect to the Harnack and Shmul'yan pre-orders, and the corresponding equivalence relations. We look at contractions that also have certain commutativity properties, for instance we consider the cases where two operators are commuting, doubly commuting or when the operators are themselves quasi-normal or hyponormal. The case of partial isometries is thoroughly analyzed and the role of commutativity properties is discussed.

\medskip

{\bf Notation and basic definitions.} In this paper $T,T' \in \B$ will be linear contractions acting on the complex Hilbert space $\h$. Also,
$V$ acting on $\ka \supset \h$ and $V'$ acting on $\ka'\supset \h'$ will denote
the \emph{minimal isometric dilations} of $T$ and $T'$ respectively. Recall (see \cite{SFb}) that $V$ is a minimal isometric dilation of $T$ if $V$ is an isometry on $\ka=\bigvee_{n\ge 0} V^n\h$ satisfying $P_{\h}V=TP_{\h}$, where $P_{\h}$ is the orthogonal projection of $\ka$ onto $\h$. The symbols $\n(T)$ and $\R(T)$ stand for the \emph{kernel} and respectively
the \emph{range} of $T$, while $I_{\h}$ denotes the identity operator on $\h$.
We simply use $I$ if the Hilbert space is clear from the context. From time to time we will consider the more general case of contractions between different Hilbert spaces. For such a contraction $T\in \mathcal{B}_1(\h_1,\h_2)$ between two Hilbert spaces $\h_1$ and $\h_2$ we denote by $D_T = (I_{\h_1}-T^{\ast}T)^{1/2}$ the \emph{defect operator} and by $\mathcal{D}_T=\overline{\R(D_T)}$ the \emph{defect space} of $T$. With some abuse of notation,
we will view the defect operator $D_T$ either as an operator on $\h_1$, on $\mathcal{D}_T$, from $\h_1$
into $\mathcal{D}_T$ or from $\mathcal{D}_T$ into $\h_1$, always using the symbol $D_T$. The precise meaning
will be clear from the context, or otherwise be made explicit.

Let $A$ be a self-adjoint operator. Then $A \ge 0$ indicates that $A$ is a \emph{positive} operator in the sense that $\langle Ax,x\rangle \ge 0$ for each $x\in \h$. If $B$ is also a self-adjoint operator, then $A \le B$, or $B \ge A$, is short for $B-A \ge 0$. For $A\in \B$ we denote by $\Re A$ the self-adjoint operator $\Re A = (1/2)(A+A^*)$. The \emph{spectrum} of $A$ is denoted by $\sigma(A)$.

Two operators $A,B \in \B$ commute if $AB=BA$ and they are said to be \emph{doubly commuting} if $A$ commutes with $B$ and its adjoint $B^*$.  We say that $A$  is \emph{quasi-normal}  if $A$ commutes with $A^*A$, that $A$ is a \emph{partial isometry} if $AA^*A = A$ and that $A$ is a \emph{quasi-isometry} if $A^*A=A^{*2}A^2$. Finally, $A$ is said to be a \emph{hyponormal} operator if $AA^* \le A^*A$.

\medskip

{\bf A review of the Harnack and Shmul'yan pre-orders.} The Harnack pre-order and equivalence relation have been introduced by Ion~Suciu \cite{S1,S2} in the 1970s based on some operator inequalities for Hilbert space contractions, which generalize the classical Harnack inequality
for positive harmonic functions in the unit disc. The corresponding equivalence classes are called Harnack parts. Recall that a \emph{pre-order} is a binary relation which is reflexive and transitive; see for instance \cite[Definition 5.2.2]{Schro}. It is well-known that given any pre-order $\prec$ on
$\mathcal{B}_1(\h)$, if one defines a binary
relation $\sim$ on $\mathcal{B}_1(\h)\times \mathcal{B}_1(\h)$
by $A\sim B$ if $A\prec B$ and $B\prec A$, then $\sim$
is an equivalence relation (cf. \cite[Proposition 5.2.4]{Schro}).

More specifically, we say that $T$ is \emph{Harnack
dominated} by $T'$ (notation $T \Prec T'$) if there exists a positive constant $c \ge 1$ such that for any analytic polynomial $p$
verifying $\Re p(z) \ge 0$ for $|z| \le 1$ we have
\begin{equation}\label{eq:definition of Harnack}
\Re p(T) \le c \Re p(T').
\end{equation}
We say that $T$ is Harnack
dominated by $T'$ \emph{with constant} $c$ whenever we want to emphasize the constant. Thus $T$ and $T'$ are
\emph{Harnack equivalent} if
$T \Prec T'$ and $T' \Prec T$; we also say in this case that $T$ and $T'$ belong to the same \emph{Harnack part}. We denote by $\Delta(T)$ the Harnack part containing the contraction $T$. The classical Harnack inequality for positive harmonic functions implies that a strict contraction $T$ (\emph{i.e.} a contraction with $\|T\| < 1$) is Harnack equivalent to the null operator $0_{\h}$. It was proved in \cite{F} that $\Delta(0_{\h})$ coincides with the class of all strict contractions.

We refer to \cite{KSS, S2, S3} for several characterizations of $T \Prec T'$. For further use we recall here the following result.

\begin{theorem}\label{thm:1.1}
 For $c\ge 1$ and two contractions $T,T' \in \B$ with minimal isometric dilations $V\in \mathcal{B}(\ka)$ and respectively $V'\in \mathcal{B}(\ka')$, the following statements are equivalent:
\begin{itemize}
 \item[(i)] $T$ is Harnack
dominated by $T'$ with constant $c^2$;
\item[(ii)] There is a unique operator $A \in \mathcal{B}(\ka',\ka)$ such that $A(\h) \subseteq \h$,
$A\mid \h = I$, $AV' = VA$ and $\|A\| \le c$.
\end{itemize}
\end{theorem}

For fixed $T$, $T'$ and $c$, the operator $A$ from (ii) intertwining the minimal isometric dilations $V'$ and $V$ of $T'$ and $T$ will be called the \emph{Harnack operator associated with} $(T,T')$.


\medskip

The Shmul'yan pre-order has been introduced in \cite{Sh} and studied in \cite{KSS} under the name $O$-pre-order.
More precisely, the Shmul'yan domination can be defined in the more general case of the unit ball
of contractions acting between two Hilbert spaces $\h_1$ and $\h_2$. Namely, if $A,B \in \mathcal{B}_1(\h_1,\h_2)$, then according to \cite{KSS} we say that $A$ {\it Shmul'yan dominates} $B$ if $B=A+ D_{A^*}XD_A$ for some $X\in \mathcal{B}(\mathcal{D}_A,\mathcal{D}_{A^*})$.  We write in this case $B \PrecS A$. Also, we say that $A$ and $B$ are {\it Shmul'yan equivalent} whenever each of them Shmul'yan dominates the other, while the corresponding equivalence classes, in $\mathcal{B}_1(\h_1,\h_2)$, are called {\it Shmul'yan parts}. We denote by $\Delta_{{\rm Sh}}(T)$ the Shmul'yan part of $T$. It was proved in \cite[Proposition 1.6 and Corollary 3.3]{KSS} that the Shmul'yan equivalence relation in the closed unit ball of $\B$ is stronger than the Harnack equivalence relation. The following theorem is inferred from some results in \cite{Sh,KSS,H}

\begin{theorem}\label{thm:1.2}
 For two contractions $A,B \in \mathcal{B}_1(\h_1,\h_2)$ the following statements are equivalent:
\begin{itemize}
 \item[(i)] $B \PrecS A$, i.e. $B=A+ D_{A^*}XD_A$ for some $X\in \mathcal{B}(\mathcal{D}_A,\mathcal{D}_{A^*})$ ;
\item[(ii)] $I - B^*A = D_{A}YD_{A}$ for some $Y \in \mathcal{B}(\mathcal{D}_A)$;
\item[(iii)] there exists $r > 0$ such that $(1-\varepsilon)A+\varepsilon B \in \mathcal{B}_1(\h_1,\h_2)$ for every $\varepsilon \in \C$ with $|\varepsilon| \le r$;
\item[(iv)] there exists $r > 0$ such that $(1-\varepsilon)A+\varepsilon B \in \mathcal{B}_1(\h_1,\h_2)$ for every $\varepsilon \in \C$ with $|\varepsilon| = r$.
\end{itemize}
Furthermore, the following statements are equivalent:
\begin{itemize}
\item[(a)] $A$ and $B$ are Shmul'yan equivalent;
\item[(b)] $B=A+D_{B^*} \widetilde{X}D_A$ for some $\widetilde{X} \in \mathcal{B}(\D_A, \D_{B^*})$;
\item[(c)] $I-A^*B=D_A\widetilde{Y}D_B$ for some $\widetilde{Y} \in \mathcal{B}(\D_B, \D_A)$.
\end{itemize}
\end{theorem}

Recently, the Shmul'yan pre-order relation has been extended in \cite{H} from contractions to functions in the \emph{operator-valued Schur class}, that is to the contractive operator-valued analytic functions on $\De$.
More precisely, let $\mathcal{E}$ and $\e'$ be two separable Hilbert spaces and let $H^{\infty}(\mathcal{E}, \e')$ be the Banach space of all norm bounded analytic functions on $\De$ with values in $\mathcal{B}(\mathcal{E},\e')$. We denote $H^{\infty}=H^{\infty}(\C)$. If $H_1^{\infty}(\mathcal{E},\e')$ is the closed unit ball of $H^{\infty}(\mathcal{E},\e')$ and $F\in H_1^{\infty}(\mathcal{E},\e')$, then the associated \emph{Toeplitz operator} is the contraction $T_F$ between the vector valued Hardy spaces $H^2(\mathcal{E})$ and $H^2(\e')$ defined by
$$
(T_Fg)(\lambda)=F(\lambda )g(\lambda )\quad (g\in H^2(\mathcal{E}), \lambda \in \De).
$$
For $F,G \in H_1^{\infty}(\mathcal{E},\e')$ we say that $F$ \emph{is Shmul'yan-ter Horst dominated by} $G$, in notation $F \PrecH G$, whenever $T_F \PrecS T_G$, i.e. $T_F$ is Shmul'yan dominated by $T_G$ in $\mathcal{B}_1 (H^2(\e), H^2(\e'))$. We denote by $\Delta_{\infty}(F)$ the equivalence class of $F$ for the equivalence relation induced by
the Shmul'yan-ter Horst pre-order.

\medskip

{\bf Organization of the paper.} The plan of the paper is as follows. The undefined terms will be defined later on.

In Section \ref{sec:2} we show that two commuting contractions $T$ and $T'$ with $T\Prec T'$ have
 the same $C_{1\cdot}$-part in their canonical triangulation, while the $C_{0 \cdot}$-parts are in the same relation of Harnack domination (Theorem \ref{te21}). Also, we study the relation between the reducing isometric parts (Theorem \ref{te24}). We exhibit some commuting conditions under which the Harnack equivalence of two given contractions $T, T' \in \B$ with $T$ quasi-normal is equivalent to their Shmul'yan equivalence and also to the existence of an arc joining $T$ and $T'$ in the class of $\mathcal{B}(\h)$-valued contractive analytic functions on the unit disc (Theorem \ref{te210}).
As a consequence, we obtain in Corollary \ref{co213} that if $T$ and $T'$ are doubly commuting contractions on $\h$ such that $T$ is quasi-normal and $T\Prec T'$, then $T$ and $T'$ are Shmul'yan equivalent.
This fact generalizes and improves a result from\cite{AST}, which asserts that two commuting normal contractions in Harnack domination are, in fact, Harnack equivalent.
In Section \ref{sec:3} we describe the Shmul'yan equivalence of two contractions acting between two Hilbert spaces in terms of Schur class functions and of the Kobayashi pseudo-distance (Theorem \ref{te31nou}). We apply this result to some classes of contractions to obtain necessary and/or sufficient conditions for Shmul'yan equivalence. We discuss here partial isometries, quasi-isometries and quasi-normal contractions (Propositions \ref{pr35nou}, \ref{pr37} and \ref{pr311}). Also, we show that the corresponding Harnack and Shmul'yan parts of a partial isometry coincide (Theorem \ref{pr38}).

 In Section \ref{sec:4} we study some properties of the
Shmul'yan-ter Horst relation.  In Theorems \ref{te44} and \ref{te45} we give a description of the corresponding equivalence class $\Delta_{\infty}(T)$ of a partial isometry $T$, viewed as a constant function in $H^{\infty}(\mathcal{E},\e')$.
\medskip

{\bf Acknowledgments.}
Both authors have been supported by the ``Laboratoire Europ\'een Associ\'e CNRS Franco-Roumain'' Math Mode. The first author has also been partially supported by the Labex CEMPI  (ANR-11-LABX-0007-01) and by the EU IRSES grant AOS (PIRSES-GA-2012-318910). The second author has also been partially supported by CNRS, France, during a three months visit to Lille. The paper was completed while the first author was in residence at MSRI (Berkeley, California) during the Fall 2016 semester; he acknowledges support from NSF Grant No. DMS-1440140. The authors are grateful to the referee for many constructive remarks, suggestions and corrections.

\section{Commuting contractions and Harnack domination}\label{sec:2}
\medskip

{\bf Canonical triangulations.} Let $T$ be a contraction acting on a Hilbert space $\h$. The \emph{asymptotic limit} $S_T \in \B$ of the
 contraction $T$ acting on a Hilbert space $\h$ is the strong limit of the sequence $\{T^{*n}T^n\}_{n\in \N}$; see, for instance,~\cite[Chapter 3]{K}.
 The asymptotic limit $S_T$ is a positive contraction and $\|S_T\|=1$ whenever $S_T\neq 0$. Notice that $\n(I-S_T)=\bigcap_{n\ge1} \n(I-T^{*n}T^n)$ is the maximal $T$-invariant subspace of $\h$ on which $T$ is an isometry, called the \emph{invariant isometric part} of $T$ in $\h$.

We say that $T$ is \emph{strongly stable} if the sequence
 $\{T^n\}_{n\in \N}$ is strongly convergent to $0$ in $\B$.
 Also, we say that $T$ is of \emph{class} $C_{0\cdot}$ (respectively, $C_{\cdot 0}$) when $T$ ($T^*$) is strongly stable. In terms of the asymptotic limit this means that $S_T=0$
 (respectively, $S_{T^*}=0$). We say that $T$ is of \emph{class} $C_{00}$ if it is of class
 $C_{0\cdot}$ and of class $C_{\cdot 0}$ and that $T$ is of \emph{class} $C_{1\cdot}$
(respectively, $C_{\cdot 1}$) if $T^nh \nrightarrow 0$ (respectively $T^{*n}h \nrightarrow 0$)
for each non-zero $h\in \h$. Also, $T$ is
of \emph{class} $C_{11}$ if both $T$ and $T^*$ are of class $C_{1\cdot}$.

Following the terminology of \cite{SFb}, the \emph{canonical triangulation} of $T$ is the matrix representation of $T$ with respect to the decomposition $\h=\n(S_T) \oplus \overline{\R(S_T)}$. Here $\n(S_T)$ is the maximum invariant subspace for $T$ on which the operator $T$ is of class $C_{0\cdot}$, while $\overline{\R(S_T)}$ is the maximum invariant subspace for $T^*$ on which $T^*$ is of class $C_{\cdot 1}$. In terms of the operator $T$ itself, the latter condition means that the contraction $PT|_{\overline{\R(S_T)}}$ is of class $C_{1\cdot}$, where $P$ is the orthogonal projection onto $\overline{\R(S_T)}$. The canonical triangulation is not preserved by the Harnack domination, more precisely it is possible to have $T \Prec T'$ and $\n(S_T)\neq \n(S_{T'})$. However, the classes $C_{0\cdot}$, $C_{\cdot0}$ and $C_{00}$ are preserved in both senses by the Harnack domination (see \cite[Theorem 5.5]{BST}). On the other hand, the class $C_{11}$ (hence $C_{1\cdot}$ and $C_{\cdot 1}$) are not preserved by the Harnack equivalence (see \cite[Propositions 5.9 and 5.10]{BST}).

In the following result we show that under the condition of commutation $TT'=T'T$ the $C_{1\cdot}$-part in the canonical triangulation is uniquely determined.

\begin{theorem}\label{te21}
Let $T$ and $T'$ be commuting contractions on $\h$ such that $T$ is Harnack dominated by $T'$. Then the canonical triangulations of $T$ and $T'$ have the form
\begin{equation}\label{eq21}
T=
\begin{pmatrix}
Q & R\\
0 & W
\end{pmatrix},
\quad
T'=
\begin{pmatrix}
Q' & R'\\
0 & W
\end{pmatrix},
\end{equation}
where $Q,Q'$ are of class $C_{0\cdot}$ on $\n(S_T)=\n(S_{T'})$ with $QQ'=Q'Q$ and $Q \Prec Q'$, and $W$ is of class $C_{1\cdot}$ on $\overline{\R(S_T)}$ such that $QR'-Q'R=(R'-R)W$.

Moreover, we have $S_T=S_T^2$ if and only if $S_{T'}=S_{T'}^2$. In this case $S_T=S_{T'}$, $W$ is an isometry and $R=R'=0$; in addition $T$ and $T'$ are Harnack equivalent if and only if $Q'\Prec Q$.
\end{theorem}

To proceed, we need the following auxiliary result.

\begin{lemma}\label{lemma0}
Let $T$ and $T'$ be two contractions on $\h$ and suppose that $TT'=T'T$ and $T \Prec T'$. If $T$  is of class $C_{\cdot 1}$ or of class $C_{1\cdot}$, then $T'=T$.
\end{lemma}

\begin{proof}
Suppose that $TT'=T'T$ and $T \Prec T'$. It is known from \cite{S2} and \cite{KSS} that there exists an operator $B \in \mathcal{B}(l^2_{\N}(\mathcal{D}_{T'}),\h)$ such that
$$
T=T'+BJ'D_{T'}, \quad BS'=TB,
$$
where $J'$ is the canonical inclusion of $\mathcal{D}_{T'}$ into $l^2_{\N}(\mathcal{D}_{T'})$ and $S'$ is the unilateral forward shift on $l^2_{\N}(\mathcal{D}_{T'})$. Then multiplying the first relation by $T^{n-1}$ for $n\ge 1$ we infer
$$
T^n=T^{n-1}T'+BS'^{(n-1)}J'D_{T'},
$$
or equivalently
$$
T^{*n}=T'^{*}T^{*(n-1)}+D_{T'}J^{'*}S'^{*(n-1)}B^*.
$$

Since $S'^{*n} \to 0$ strongly we have $(T^*-T'^*)T^{*(n-1)} \to 0$ strongly on $\h$. Using the commutativity of $T$ and $T'$ we get
$$
T^{*(n-1)}(T^*-T^{'*})=D_{T'}J^{'*}S'^{*(n-1)}B^* .
$$
Using again that $S^{'*n} \to 0$ strongly, we obtain $S_{T^*}(T^*-T'^*)=0$. This relation shows that if $T$ is of class $C_{\cdot 1}$, that is $\n(S_{T^*})=\{0\}$, then $T=T'$. Similarly, the same conclusion holds if $T$ is of class $C_{1\cdot}$ because
$T^* \Prec T'^*$ and $T^*$ commutes with $T^{'*}$.
\end{proof}

\begin{proof}[Proof of Theorem \ref{te21}]
We notice that the commutativity condition $TT'=T'T$ implies that the subspace $\n(S_T)=\{h \in \h: T^n h \to 0\}$ is invariant for both $T$ and $T'$. Therefore, the matrix representations of $T$ and $T'$ with respect to the decomposition $\h=\n(S_T) \oplus \overline{\R(S_T)}$ are of the form
$$
T=
\begin{pmatrix}
Q & R \\
0 & W
\end{pmatrix},
\quad T'=
\begin{pmatrix}
Q' & R'\\
0 & W'
\end{pmatrix},
$$
where $QQ'=Q'Q$, $WW'=W'W$ and $QR'+RW'=Q'R+R'W$. Clearly, the matrix of $T$ is just its canonical triangulation. Since the joint closed invariant subspaces conserve the Harnack domination (see \cite[Lemma 2.2]{BST}), one has $Q \Prec Q'$ and $W^* \Prec W'^*$, hence $W \Prec W'$. Also, $Q'$ is of class $C_{0\cdot}$ because $Q$ is such (see \cite[Theorem 5.5]{BST}). Therefore $\n(S_T) \subset \n(S_{T'})$. Lemma \ref{lemma0} implies that $W=W'$ since $W$ is of class $C_{1\cdot}$ by the properties of the canonical triangulation. Keeping in mind the inclusion $\n(S_T) \subset \n(S_{T'})$, the condition $W=W'$ forces the equality $\n(S_T)=\n(S_{T'})$. Therefore we have $\overline{\R(S_T)} =\overline{\R(S_{T'})}$. Hence the above matrix representation of $T'$ is exactly its canonical triangulation. The first assertion of theorem is proved.

Let us assume now that $S_T=S_T^2$. Then $S_T$ is an orthogonal projection onto $\R(S_T)=\n(I-S_T)$, a subspace which is reducing for $T$. Since $W|_{\R(S_T)}$ is an isometry and $T'^*T'\le I$, it follows that $R'=0$ in the block matrix \eqref{eq21} of $T'$. So $\R(S_T)$ also reduces $T'$ and $T=Q \oplus W$, $T'=Q'\oplus W$. As $Q$ and $Q'$ are of class $C_{0\cdot}$, we have $S_{T'}=S_{T'}^2=S_T$.

Conversely, suppose that $S_{T'}=S_{T'}^2$. Then from \eqref{eq21} we have that $W=T'|_{\R(S_T)}$ is an isometry. Therefore $\n(I-S_T)\subset \overline{\R(S_T)}=\n(I-S_{T'})$. As the reverse inclusion also holds (we use here the relation $T \Prec T'$ and \cite[Lemma 5.1]{BST}), we obtain $\n(I-S_T)=\R(S_T)=\n(I-S_{T'})$. We conclude that $S_T=S_T^2$ and, in this case, the relation of Harnack equivalence between $T$ and $T'$ reduces to the relation $Q'\Prec Q$. This ends the proof.
\end{proof}

We infer the following consequence, completing Lemma \ref{lemma0}.

\begin{corollary}\label{co22nou}
Let $T$ and $T'$ be commuting contractions on $\h$ such that $T$ is Harnack dominated by $T'$. If $T'$ is of class $C_{1\cdot}$ or of class $C_{\cdot 1}$, then $T'=T$.
\end{corollary}

\begin{proof}
If $T'$ is of class $C_{1\cdot}$ (or $C_{\cdot1}$), then the conclusion $T=T'$ follows from the assertion of Theorem \ref{te21}. Indeed, in this case $T'=W=T$ on $\h=\overline{\R(S_{T'})}=\overline{\R(S_T)}$.
\end{proof}

\begin{remark}\label{re22}
\rm
A different proof of Theorem \ref{te21} can be given starting from the canonical triangulation of $T'$ with respect to the decomposition $\h=\n(S_{T'}) \oplus \overline{\R(S_{T'})}$. Indeed, in this case $\n(S_{T'})$ is invariant for $T$, and using  Theorem \ref{te21}, the operator $T$ will have a matrix representation of the form given in Theorem \ref{te21}. As $Q \Prec Q'$ and $Q'$ is of class $C_{0\cdot}$, $Q$ will be also of class $C_{0\cdot}$ on $\n(S_{T'})$ (see  \cite[Theorem 5.5]{BST}). Therefore $\n(S_{T'})\subset \n(S_T)$. But $T^*=T'^*=W^*$ on $\overline{\R(S_{T'})}\supset \overline{\R(S_T)}$ and $W$ is of class $C_{1\cdot}$, so $\n(S_T)=\n(S_{T'})$.
\end{remark}

\begin{corollary}\label{co23}
Let $T$ and $T'$ be commuting contractions on $\h$ such that $T$ is Harnack dominated by $T'$. Then $\n(I-S_T)=\n(I-S_{T'})$ if and only if $\n(I-S_T)$ is invariant for $T'$. This occurs, in particular, when $T'$ and $S_T$ commute.
\end{corollary}

\begin{proof}
As we have already mentioned, the inclusion $\n(I-S_{T'}) \subset \n(I-S_T)$ always holds and $T=T'$ on $\n(I-S_{T'})$. Now if $\n(I-S_T)$ is invariant for $T'$, then the relation $T\Prec T'$ is also true on this subspace, while by the assertion of Lemma \ref{lemma0} we have $T=T'$ on $\n(I-S_T)$. So, $T'$ is an isometry on this subspace, which shows that $\n(I-S_T)=\n(I-S_{T'})$. The other implication is obvious.
\end{proof}

{\bf Reducing isometric and unitary parts.} Let $T$ be a contraction on $\h$. Recall that the subspace $\h_u = \n(I-S_T)\cap \n(I-S_{T^*})$ is the maximal reducing subspace for $T \in \B$ on which $T$ is unitary, called the \emph{reducing unitary part} of $T$ (see \cite{K}).
The \emph{reducing isometric part} in $\h$ of $T$, denoted $\h_i$, is the maximal reducing subspace for $T$ on which the operator $T$ is an isometry. The existence of $\h_i$ is a consequence of Zorn's Lemma; see \cite[Proposition 2.1]{LS}.
We have $\h_u \subset \h_i \subset \n(I-S_T)$. According to \cite[Proposition 2.8]{LS}, the reducing isometric part in $\h$ of $T$ can be described
using the minimal isometric dilation $V$ of $T$ as
\begin{equation}\label{ec11}
\h_i=\h \ominus \bigvee_{n,j\ge0} T^n(I-T^{*j}T^j)\h=\left\{h\in \h: V^nT^{*m}h\in \h, \hspace*{1mm} m,n \ge 0\right\} .
\end{equation}

Corollary \ref{co23} gives conditions under which the invariant isometric part of a contraction is preserved by Harnack domination. In some cases even the reducing isometric part can be conserved by Harnack domination, as follows from the last assertion of Theorem \ref{te21}, where the corresponding asymptotic limits are orthogonal projections.

If $T \Prec T'$, then we always have the inclusion $\h_u'\subset \h_u$ between the reducing unitary parts. For the reducing isometric parts $\h_i$ and $\h_i'$ there is no relationship, in general. The following result examines the case of commuting contractions.

\begin{theorem}\label{te24}
Let $T$ and $T'$ be commuting contractions on $\h$ such that $T$ is Harnack dominated by $T'$, and let $\h_i$ and $\h_i'$ be their reducing isometric parts, as above. Then:
\begin{itemize}
\item[(i)] $\h_i'\subset \h_i$ and $T'=T$ on $\h_i'$.

\item[(ii)] $\h_i'$ reduces $T$ if $\h_i \subset \n(D_{T'})$.
In addition, if $\n(I-S_{T})$ reduces $T'$, then
$$
\h_i=\h_i'=\n(I-S_{T'})=\n(I-S_T).
$$
\item[(iii)] If the Harnack operator associated to the pair $(T,T')$ is injective, then $\h_i'$ reduces $T$.
\item[(iv)] If $T$ and $T'$ are Harnack equivalent, then $\h_i=\h_i'$.
\end{itemize}
\end{theorem}

\begin{proof}
To see the inclusion in (i) we use the descriptions \eqref{ec11} for $\h_i$ and $\h_i'$ with respect to the minimal isometric dilations $V\in \mathcal{B}(\ka)$ and $V'\in \mathcal{B}(\ka ')$ of $T$ and $T'$, respectively. Since
$$
\h_i' \subset \n(I-S_{T'}) \subset \n(I-S_T) \cap \n(D_{T'}),
$$
we have $T=T'$ on $\h_i'$. Therefore $\h_i'$ is invariant for $T$ and $T$ is an isometry on $\h_i'$.

Now let $h\in \h_i'$. Then, using \eqref{ec11} for $\h'_i$, we obtain $V^{'n}T'^{*m}h\in \h$ for $m,n \in \N$. Let $A\in \mathcal{B}(\ka',\ka)$ denote the Harnack operator associated with $(T,T')$ as in Theorem \ref{thm:1.1}. In particular, $A$ satisfies $AV'=VA$ with $A|_{\h}=I_{\h}$. Since $TT'=T'T$ we obviously have $T^mh=T^{'m}h$, so $T^m$ is an isometry on $\h'_i$. Therefore
\begin{eqnarray*}
V^nT^{*m}h&=& V^nT^{*m}T'^{*m}T'^mh=V^nAT'^{*m}T^{*m}T^mh\\
&=& AV'^{n}T'^{*m}h=V'^nT'^{*m}h\in \h.
\end{eqnarray*}
Hence it follows from \eqref{ec11} that $h\in \h_i$ which proves the inclusion $\h_i'\subset \h_i$. This gives a proof of statements (i) and (iv).

Suppose now that $\h_i \subset \n(D_{T'})$. Since $T=T'$ on $\n(D_{T'})$ by the relation $T \Prec T'$, for $h\in \h_i'$ and $k\in \h_i \ominus \h_i'$ we have $T^*h\in \h_i$, $T'^*h\in \h_i'$ and
$$
\langle T^*h,k \rangle =\langle h, T'k\rangle =\langle T'^*h,k\rangle =0.
$$
Therefore $T^*h\in \h_i'$ whenever $h\in \h_i'$, and consequently $\h_i'$ reduces $T$. Clearly, one has $\h'_i \subset \h_i \subset \n(I-S_T)$ and if  $\n(I-S_{T})$ reduces $T'$, then by Corollary \ref{co23} we have $\n(I-S_T)=\n(I-S_{T'})$. So, this subspace reduces $T'$ to an isometry, hence it is contained in $\h'_i$, by the maximality property of this later subspace. Finally, one obtains the equalities
$$
\h_i'=\h_i=\n(I-S_T)=\n(I-S_{T'}).
$$
Thus, the statements in (ii) are proved.

For (iii) we assume that the Harnack operator $A$ is injective.  Using \eqref{ec11} for $\h_i$ and the inclusion $T'^*h\in \h_i'\subset \h_i$, we have for $h\in \h_i'$ and $m,n\in \N$ that
$$
AV'^nT'^{*m}T^*h=V^nT^*T'^{*m}h \in \h.
$$
Let $h_{n,m}$ denote the common value of this element. As $A|_{\h}=I$, one has $AV'^nT'^{*m}T^*h=h_{n,m}=Ah_{n,m}$ and since $A$ is injective it follows that $V'^nT'^{*m}T^*h=h_{n,m}$; hence $T^*h\in \h_i'$. Therefore $\h_i'$ is invariant for $T^*$. As $T\h_i'=T'\h_i'\subset \h_i'$, the subspace $\h_i'$ reduces $T$. The statement (iii) follows, and this completes the proof.
\end{proof}

\begin{remark}\label{re26}
\rm
Suppose that  $T \Prec T'$. It is easy to see that if $\h_i=\h_i'$ then $\h_u=\h_u'$, but the converse implication is not true, in general. In fact, it was proved in \cite[Example 6.2 and Proposition 6.3]{BST} that there exists a Harnack part which contains a completely non-isometric contraction and another completely non-unitary contraction which has non-trivial reducing isometric part.
Moreover, such contractions can Harnack dominate a unitary operator (in this case, the right bilateral shift on $l^2_{\Z}(\mathcal{E})$ for a Hilbert space $\mathcal{E}$). Hence, one has $\h_u'\subset \h_u$ whenever $T \Prec T'$ but there are examples of Harnack equivalent contractions $T$ and $T'$ such that $\h_i' \nsubseteq \h_i$ (see Example \ref{ex36} below). Theorem \ref{te24} (iv) shows that $\h_i' = \h_i$ under the condition of commutativity of $T$ and $T'$.
\end{remark}

Concerning the unitary parts in $\h$ of $T$ and $T'$ (or the reducing isometric parts of $T^*$ and $T^{'*}$) we obtain in a particular case the following consequence.

\begin{corollary}\label{co27}
Let $T$ and $T'$ be commuting contractions on $\h$ such that $T$ is hyponormal and
$T \Prec T'$. Then $T$ and $T'$ have the same unitary part.
\end{corollary}

\begin{proof}
Indeed, it is known (see \cite{K}) that $S_{T^*}=S_{T^*}^2$ for $T$ hyponormal. Therefore $\n(I-S_{T^*}) \subset \n(I-S_T)$ and by the last assertion of Theorem \ref{te21}, or by the second assertion of Theorem \ref{te24} (ii), we have
$$
\h_u=\n(I-S_{T^*})=\n(I-S_{T'^*})=\h_u'.
$$
Thus $T$ and $T'$ have the same unitary part.
\end{proof}

\bigskip

{\bf When Harnack and Shmul'yan equivalences are the same.} The next result exhibits some commutativity conditions under which the Harnack equivalence of two given contractions $T, T' \in \B$ is equivalent to their Shmul'yan equivalence (see Theorem \ref{thm:1.2}) and also with the existence of an arc joining $T$ and $T'$ in the class of $\mathcal{B}(\h)$-valued contractive analytic functions on the unit disc. Recall that a quasi-normal contraction is a contraction $T$ commuting with $T^*T$.

\begin{theorem}\label{te210}
Let $T,T'$ be contractions on $\h$ such that $T^*$ is quasi-normal and $T$ commutes with $T'$ and $T'^*T'$, while $T'$ commutes with $TT^*$. The following are equivalent :
\begin{itemize}
\item[(i)] $T$ is Harnack dominated by $T'$;

\item[(ii)] $T$ is Harnack equivalent to $T'$;

\item[(iii)] $T$ is Shmul'yan equivalent to $T'$;

\item[(iv)] There exists $z_0\in \De$ and a $\B$-valued contractive analytic function $F$ on $\De$ such that $F(0)=T$ and $F(z_0)=T'$;

\item[(v)] There exists an operator $B\in \mathcal{B}(l^2_{\N}(\mathcal{D}_{T'}),\h)$ such that
\begin{equation}\label{ec22}
T=T'+BJ'D_{T'}, \quad BS'=TB,
\end{equation}
where $J'$ is the canonical inclusion of $\mathcal{D}_{T'}$ into $l^2_{\N}(\mathcal{D}_{T'})$ and $S'$ is the forward shift on $l^2_{\N}(\mathcal{D}_{T'})$.
\end{itemize}
\end{theorem}

\begin{proof}
The plan of the proof is the following chain of implications:
$$ {\rm (i)} \Longrightarrow {\rm (v)} \Longrightarrow {\rm (iii)} \Longrightarrow {\rm (iv)} \Longrightarrow {\rm (ii)} \Longrightarrow {\rm (i)}.$$

Suppose that (i) holds, that is $T$ is Harnack dominated by $T'$. Let $V$ and $V'$ be the minimal isometric dilations of $T$ and $T'$, acting on $\ka=\h\oplus l^2_{\N}(\mathcal{D}_T)$ and $\ka'=\h\oplus l^2_{\N}(\mathcal{D}_{T'})$, respectively. With respect to these decompositions, the matrix representations of $V$ and $V'$ are given by
\begin{equation}\label{ec24}
V=
\begin{pmatrix}
T & 0\\
JD_T & S
\end{pmatrix},
\quad
V'=
\begin{pmatrix}
T' & 0\\
J'D_{T'} & S'
\end{pmatrix},
\end{equation}
where $J$ and $S$ have similar definitions as the operators $J'$ and $S'$, respectively, defined in \eqref{ec22}. Let $A$ be the Harnack operator associated with $T$ and $T'$. Then $A$ verifies $A|_{\h}=I_{\h}$ and $AV'=VA$. This last equality yields
$$
P_{\h}AV'=P_{\h}VP_{\h}A=TP_{\h}A
$$
because $V(\ka \ominus \h)\subset \ka \ominus \h$. On the other hand, since $V'h=T'h\oplus J'D_{T'}h$ for $h\in \h$, we get
$$
Th=P_{\h}VAh=P_{\h}AV'h=T'h+P_{\h}AJ'D_{T'}h.
$$
Therefore the operator $B:=P_{\h}A|_{l^2_{\N}(\mathcal{D}_{T'})}$ satisfies the identities of (v).
This shows that (i) implies (v).

Suppose now that (v) holds true. Denote $B_0=BJ'\in \mathcal{B}(\mathcal{D}_{T'},\h)$. Then the second relation in \eqref{ec22} gives
$$
B= [B_0,TB_0,T^2B_0,...],
$$
so one can consider the positive operator $Z \in \B$ given by
$$
Z=\sum_{n=0}^{\infty} T^nB_0B_0^*T^{*n}.
$$
Since by hypothesis $T$ commutes with $T'^*T'$ it follows that $\mathcal{D}_{T'}$ reduces $T$. Set $T_1=T|_{\mathcal{D}_{T'}}$. As $T$ and $T'$ commute we have
$$
TB_0D_{T'}=T(T-T')=(T-T')T=B_0D_{T'}T=B_0T_1D_{T'},
$$
that is $TB_0=B_0T_1$. Because $T^*$ is quasi-normal, $T_1^*$ will be quasi-normal too. Furthermore, since $T'$ commutes with $TT^*$, one can see, as above, that $TT^*B_0=B_0T_1T_1^*$. Therefore for each integer $n\ge 0$ we get
$$
T^{n+1}B_0B_0^*T^{*(n+1)}=T^nB_0T_1T_1^*B_0^*T^{*n}=T^nTT^*B_0B_0^*T^{*n}=TT^*T^nB_0B_0^*T^{*n},
$$
and by taking the adjoint $T^{n+1}B_0B_0^*T^{*(n+1)}=T^n B_0B_0^*T^{*n}TT^*$. Applying this to each of the summands in $TZT^*=\sum_{n=0}^{\infty} T^{n+1}B_0B_0^*T^{*(n+1)}$ we obtain the identities $TZT^*=TT^*Z=ZTT^*$. Now, the identity $B_0B_0^*=D_{T^*}^2Z$ follows from the fact that $Z$ satisfies the equation $Z=TZT^*+B_0B_0^*$. Also, $TT^*Z=ZTT^*$ implies that $ZD_{T^*}^2=D_{T^*}^2Z$ and hence $ZD_{T^*}=D_{T^*}Z$. This yields $B_0B_0^*=D_{T^*}^2Z=D_{T^*}ZD_{T^*}$, which (by Douglas's Lemma in \cite{D}) implies that $B_0=D_{T^*}W$ for some operator $W \in \mathcal{B}(\mathcal{D}_{T'},\mathcal{D}_{T^*})$. This relation, together with the first identity from (v), yields
\begin{equation}\label{ec23}
T=T'+D_{T^*}WD_{T'}.
\end{equation}
Thus $T$ is Shmul'yan equivalent to $T'$ (by Theorem \ref{thm:1.2}) and consequently (v) implies (iii).

The implication from (iii) to (iv) follows from \cite[Theorem 4.1]{S4} and Theorem \ref{thm:1.2} (and also from Theorem \ref{te31nou} below).

It was proved in \cite[Theorem 10]{Sh} and \cite[Theorem 4]{S3} (see also Remark \ref{re212} below) that the range of a contractive analytic function is contained in a Harnack part, hence the contractions $T$ and $T'$ satisfying (v) are Harnack equivalent. Therefore (iv) implies (ii), and (ii) obviously implies (i). This ends the proof of the theorem.
\end{proof}

We would like to emphasize the role played by commutativity properties in the proof of the implication (v) $\Rightarrow$ (iii).


 \begin{remark}\label{re211'}
 \rm
 Recall that $T\Prec T'$ if and only if $T^* \Prec T'^*$. Therefore a dual version of Theorem \ref{te210} holds true, assuming this time that $T$ is quasi-normal and that the corresponding commuting relations are satisfied, that is $TT'=T'T$, $TT'T'^*=T'T'^*T$, $T'T^*T=T^*TT'$.
 \end{remark}

\begin{remark}\label{re212}
\rm


By \cite[Theorem 4.1]{S4} and the proof of \cite[Theorem 4]{S3}, any Shmul'yan part is a union of ranges of $\B$-valued contractive analytic functions on $\De$. Thus, Theorem \ref{te210} asserts that, under some additional commutativity conditions, a contraction $T'$ which Harnack dominates a quasi-normal $T$ belongs together with $T$ to the range of an analytic functions on $\De$, which is contained in the Shmul'yan part of $T$.
\end{remark}

It has been proved in \cite{AST} that two commuting normal contractions in Harnack domination are, in fact, Harnack equivalent. Since by the Fuglede-Putnam theorem (see \cite{K}) a pair of commuting normal operators is doubly commuting, the result from \cite{AST} (and even more) is now directly obtained from Theorem~\ref{te210}. Another direct consequence of the same Theorem \ref{te210} is the following effective generalization of the result from \cite{AST}.

\begin{corollary}\label{co213}
Let $T$ and $T'$ be doubly commuting contractions on $\h$ such that $T$ is quasi-normal and it is Harnack dominated by $T'$. Then $T$ and $T'$ are Shmul'yan equivalent.
\end{corollary}

Notice that this corollary is not true if the role of $T$ and $T'$ is interchanged, that is when $T'$ is quasi-normal and $T$ is not. For example, the null contraction Harnack dominates any contraction $T$ with $\|T\|=1$ and spectral radius less than one, but the null operator and $T$ are not Harnack equivalent by Foias's result in \cite{F}.

Returning to the assertions contained in Theorem \ref{te21}, we note that they essentially depend on the commutation of $T$ with $T'$. We illustrate this fact by the example below which discuss the role of the commutation properties. This example also shows that the inclusion $\h_i' \subset \h_i$ can be strict under Harnack equivalence (see also Remark \ref{re26}).

\begin{example}\label{ex36}
Let $\mathcal{E}$ be a complex Hilbert space. Consider, for $\lambda \in \C$ with $|\lambda |\le 1$, the weighted forward shift $T(\lambda)$ on $l^2_{\N}(\mathcal{E})$ given by
$$
T(\lambda) (e_0,e_1,...)=(0,\lambda e_0, e_1,...) , \quad e=(e_n)_{n\in \N} \in l^2_{\N}(\mathcal{E}).
$$
Then its adjoint is given by $T(\lambda)^*e=(\overline{\lambda}e_1,e_2,...)$,
and for $n\ge 1$ we have
$$
T(\lambda)^{*n}T(\lambda)^ne=(|\lambda|^2e_0,e_1,...)=S_{T(\lambda)}e.
$$

Clearly, $T(\lambda)$ is hyponormal and, for $\lambda \neq 0$, the operator $S_{T(\lambda)}$ is invertible. Therefore $T(\lambda)$ is of class $C_{1\cdot}$ and $T(\lambda)^*$ is of class $C_{0\cdot}$. Moreover, the operators $T(\lambda)$ are all Harnack equivalent for $|\lambda|<1$ and $T(\lambda)$ Harnack dominates the isometry $T(\mu)$ whenever $|\mu |=1$. These facts can be proved using the same arguments as in the proof of \cite[Example 6.1 and Proposition 6.4]{BST}. We remark that $T(\lambda)$ and $T(\mu)$ does not commute for such $\lambda$'s and $\mu$'s. This shows that the first assertion of Theorem \ref{te21} and Corollary \ref{co22nou} can be false without the commutativity condition.

On the other hand, $T=T(0)$ has the property that $S_T=S_T^2$, so the same remark holds concerning the second assertion of Theorem \ref{te21}. In addition, we have
$$
\n(I-S_T)=\n(I-S_{T(\lambda)})=\{(e_n)_{n\in \N} \in l^2_{\N}(\mathcal{E}):e_0=0\} \quad (0<|\lambda |<1).
$$
Hence
$$
\h_i=\h_i(T)=\n(I-S_T) \subsetneqq \h =\h_i(T(\mu)) \quad (|\mu |=1)
$$
and also
$$
\n(S_T)=\n(T)=\n(I-S_{T(\lambda)})^{\perp}=\{(e_0,0,...):e_0\in \mathcal{E}\} \quad (0<|\lambda |<1).
$$
In fact, we have the decomposition $T=S\oplus 0$ on $l^2_{\N}(\mathcal{E})=\n(I-S_T)\oplus \n(S_T)$, where $S$ is a shift on $\n(I-S_T)$. Hence $T$ is quasi-normal. Notice that this orthogonal decomposition of $l^2_{\N}(\mathcal{E})$ is not reducing for the $C_{10}$-contractions $T(\lambda)$ with $\lambda \neq 0$, even in the case $|\lambda |=1$ when $S_{T(\lambda)}=I$. Therefore the condition $S_T=S_T^2$ does not imply that $\n(S_{T'})$ is reducing for $T$, in general. This has to be compared with Theorems \ref{te21} and \ref{te210}).

Now let $T'\neq T(\lambda)$ be another contraction acting on $l^2_{\N}(\mathcal{E})$ which belongs to the Harnack part $\Delta(T)$ of $T$. Then $\n(I-S_{T'})=\n(I-S_T)$, hence
$$
\n(S_{T'})\subset \n(I-S_{T'})^{\perp} =\n(S_T).
$$
We distinguish two possible cases: either $\n(S_{T'})=\{0\}$ (that is, $T'$ is of class $C_{10}$ like $T$), or $\n(S_{T'})\neq \{0\}$. When $\textrm{dim} \, \mathcal{E} >1$ and $S_{T'}\neq S_{T'}^2$ we have $\n(S_{T'})\neq \n(S_T)$. In this case $T'^*|_{\n(S_T)} \echiH 0$; therefore
$
\|P_{\n(S_T)}T'|_{\n(S_T)}\|<1.
$

Suppose now that $\mathcal{E}=\C$. If $\n(S_{T'})\neq \{0\}$, then $\n(S_{T'})=\n(S_T)$, so $S_{T'}=S_{T'}^2$. In this case, any contraction $T'$ in the Harnack part $\Delta(T)$ of $T$ is either a $C_{10}$-contraction, or we have $T'=S\oplus Q'$ with $S=T|_{\n(I-S_T)}$ an isometry and $\|Q'\|<1$ on $\n(S_{T})$. Then $S_{T'}=S_T$, therefore some assertions of Theorem \ref{te21} can be true without the commutativity condition.

Notice also that in the Harnack part $\Delta(T)$ of $T=T(0)$ the only quasi-normal partial isometry is $T$ itself. The other contractions in $\Delta(T)$ are either hyponormal (even subnormal) of class $C_{10}$
(as $T(\lambda)$ with $0<|\lambda |<1$), or do not belong to these classes (as $T'$ before).
Also, notice that the reducing isometric part is not preserved by the Harnack equivalence, in general.
Indeed, $\n(I-S_T)$ is only invariant for $T(\lambda)$ when $\lambda \neq 0$. In fact, the reducing isometric parts of $T=T(0)$ and $T'=T(\lambda)$ with $\lambda \neq 0$ are given by $\h_i=\n(I-S_T)$ and $\h_i'=\{0\}$.

\end{example}

\medskip

\section{Remarks on Shmul'yan and Harnack equivalences}\label{sec:3}
\medskip

We consider now the Shmul'yan equivalence of contractions acting between two separable Hilbert spaces $\e$ and $\e'$. Our first aim is to give a generalization in this context of a result from \cite{S4} obtained for contractions acting on same space.

\bigskip

{\bf Shmul'yan parts and the Kobayashi pseudo-distance.}
Recall that the {\it hyperbolic metric} on $\De$ is defined by
$$
\delta (z, \mu)=\tanh ^{-1} (|\frac{z-\mu }{1-\overline{\mu}z}|), \quad z,\mu \in \De.
$$
According to the terminology of \cite{S4} (see also \cite{Di}), the \emph{Kobayashi pseudo-distance} on $\mathcal{B}_1(\e,\e')$ is the mapping
$$\delta_K:\mathcal{B}_1(\e,\e') \times \mathcal{B}_1(\e,\e') \to [0,\infty]; \quad \delta_K(T,T'):= \inf_{\lambda_j,F_j} \sum_{j=1}^n \delta (0,\lambda _j) .$$
Here the infimum is taken over all finite systems $\{z_j\}_1^n\subset \De$ for which there exists $\{F_j\}_1^n \subset H_1^{\infty} (\e,\e')$ such that
    $$
    F_1(0)=T, \quad F_j(\lambda _j)=F_{j+1}(0) \quad {\rm for} \quad 1\le j\le n-1, \quad F_n(\lambda _n)=T'.
    $$
The Kobayashi pseudo-distance takes the value $\infty$ when such $\{z_j\}$'s and $\{F_j\}$'s do not exist.  This pseudo-distance $\delta_K$ has all the properties of a metric except the fact that $\delta_K(T,T')=0$ implies $T=T'$.

\begin{theorem}\label{te31nou}
Let $T, T'$ be two contractions from $\e$ into $\e'$. The following are equivalent:
\begin{itemize}
\item[(i)] $T$ and $T'$ are in the same Shmul'yan part;
\item[(ii)] There exists $\lambda _0\in \D$ and $F\in H_1^{\infty} (\e,\e')$ such that $F(0)=T$ and $F(\lambda_0)=T'$;
\item[(iii)] $\delta_K(T,T')< \infty$.
\end{itemize}
\end{theorem}

\begin{proof}
Let $T,T' \in \mathcal{B}_1 (\e,\e')$ be Shmul'yan equivalent. Define the contractions $\widetilde{T}, \widetilde{T}'$ on $\e \oplus \e'$ by
$$
\widetilde{T}(e\oplus e')=0\oplus Te, \quad \widetilde{T}'(e\oplus e') =0 \oplus T'e, \quad e\in \e, \quad e'\in \e'.
$$
Then $D_{\widetilde{T}^*}=I_{\e}\oplus D_{T^*}$ and $D_{\widetilde{T}'}=D_{T'}\oplus I_{\e'}$ on $\e \oplus \e'$.

Since $T$ and $T'$ are Shmul'yan equivalent, by Theorem \ref{thm:1.2} there exists an operator $X\in \mathcal{B}(\D_{T'}, \D_{T^*})$ such that $T=T'+D_{T^*}XD_{T'}$. Defining $\widetilde{X} \in \mathcal{B}(\e \oplus \e')$ by the block matrix
$$
\widetilde{X}=
\begin{pmatrix}
0 & 0\\
X & 0
\end{pmatrix}
$$
we have
\begin{eqnarray*}
D_{\widetilde{T}^*} \widetilde{X}D_{\widetilde{T}'}
&=&
\begin{pmatrix}
I_{\e} & 0\\
0 & D_{T^*}
\end{pmatrix}
\begin{pmatrix}
0 & 0\\
X & 0
\end{pmatrix}
\begin{pmatrix}
D_{T'} & 0\\
0 & I_{\e'}
\end{pmatrix}\\
&=&
\begin{pmatrix}
0 & 0\\
D_{T^*}X D_{T'} &0
\end{pmatrix}
=
\begin{pmatrix}
0 & 0\\
T-T' & 0
\end{pmatrix}\\
&=& \widetilde{T}-\widetilde{T}'.
\end{eqnarray*}
Denoting $Y=P_{\D_{\widetilde{T}^*}}\widetilde{X} |_{\D_{\widetilde{T}'}}$ it follows that $\widetilde{T}=\widetilde{T}'+D_{\widetilde{T}^*}YD_{\widetilde{T}'}$. Thus $\widetilde{T}$ and $\widetilde{T}'$ are Shmul'yan equivalent in $\mathcal{B}_1(\e\oplus \e')$. Then by \cite[theorem 4.1]{S4} there exists $\lambda _0\in \De$ and $G\in H_1^{\infty} (\e \oplus \e')$ such that $G(0)=\widetilde{T}$, $G(\lambda _0)=\widetilde{T}'$. Now the analytic function $F: \De \to \mathcal{B}_1(\e,\e')$ given by
$$
F(\lambda)=P_{\e'}G(\lambda)|_{\e}, \quad \lambda \in \De
$$
belongs to $H_1^{\infty}(\e,\e')$ and satisfies
$$
F(0)=P_{\e'}\widetilde{T}|_{\e}=T, \quad F(\lambda_0)=P_{\e'}\widetilde{T}'|_{\e}=T'.
$$
So, we proved that (i) implies (ii). Obviously, (ii) implies (iii).

Assume next (iii), \emph{i.e.} $\delta_K(T,T')<\infty$. Therefore there exist $\{\lambda_j\}_1^n\subset \De$ and $\{F_j\}_1^n\subset H_1^{\infty} (\e,\e')$ as in (iii). But by \cite[Theorem 10]{Sh} the range $G(\De)$ of any function $G\in H_1^{\infty}(\e,\e')$ is contained in a Shmul'yan part. Hence the contractions $F_j(0)$ and $F_j(\lambda_j)$ are Shmul'yan equivalent, for $1\le j\le n$, in particular $T=F_1(0)$ and $T'=F_n(z_n)$ are so. We have showed that (iii) implies (i) and this ends the proof.
\end{proof}

\begin{remark}\label{re32nou}
 Theorem \ref{te31nou} shows that the Shmul'yan parts of $\mathcal{B}_1(\e,\e')$ coincide to the class of equivalence induced by the equivalence relation : $T \echiK T'$ whenever $\delta_K (T,T')<\infty$. Also, in \cite[Theorem 5.1]{S4} it was proved that $\delta_K$ is a true metric on each Shmul'yan (Kobayashi) part of $\mathcal{B}_1(\e)$ and that the $\delta_K$-topology is stronger than the topology of hyperbolic metric on such parts of $\mathcal{B}_1(\e)$. An interesting problem is to see if $\delta_K$ is a complete metric on the Shmul'yan parts of $\mathcal{B}_1(\e,\e')$, as it is the case of the hyperbolic metric on the Harnack parts of $\mathcal{B}_1(\e)$ (see \cite{SV}).
\end{remark}

Returning to Theorem \ref{te31nou} we have the following
\begin{corollary}\label{co33nou}
Let $T,T' \in \mathcal{B}_1(\e,\e')$ be two contractions which, with respect to two orthogonal decompositions $\e=\e_0 \oplus \e_1$ and $\e'=\e'_0 \oplus \e'_1$, have the block matrices
\begin{equation}\label{eq31nou}
T=
\begin{pmatrix}
T_0 & T_1\\
T_2 & T_3
\end{pmatrix},
\quad
T'=
\begin{pmatrix}
T_0' & T_1'\\
T_2' & T_3'
\end{pmatrix}.
\end{equation}
Let $C_0,C_0' \in \mathcal{B}_1(\e_0,\e_0'\oplus \e'_1)$ and $C_1,C_1' \in \mathcal{B}_1(\e_1, \e'_0\oplus \e'_1)$ be the column operators of $T,T'$ and let $L_0,L'_0 \in \mathcal{B}_1(\e_0\oplus \e_1, \e'_0)$ and $L_1,L'_1\in \mathcal{B}_1 (\e_0\oplus \e_1, \e'_1)$ be the row operators of $T,T'$, respectively.

If $T$ and $T'$ are Shmul'yan equivalent then $C_j$ and $C'_j$, respectively $L_j$ and $L'_j$ are Shmul'yan equivalent, for $j=0,1$. Moreover, in this case $T_j$ and $T'_j$ are Shmul'yan equivalent, for $j=0,1,2,3$.
\end{corollary}


The following fact is an analogue of the corresponding result for Harnack domination (see \cite[Lemma 2.2]{BST}).

\begin{corollary}\label{co34nou}
Let $T,T' \in \mathcal{B}_1(\e)$ be Shmul'yan equivalent. If $\e_0\subset \e$ is a closed subspace invariant to both $T$ and $T'$, then $T|_{\e_0}$ and $T'|_{\e_0}$ are Shmul'yan equivalent.
\end{corollary}

Notice that the assertions from the above corollaries can be also derived from Theorem \ref{thm:1.2} (iii); see also \cite[Corollary 1.5]{H}. Here we reobtained these results as consequences of Theorem \ref{te31nou}.

We now see that a partial converse assertion of that in Corollary \ref{co33nou} holds in some particular cases, even without the separability condition on $\e$ and $\e'$.

\begin{proposition}\label{pr35nou}
Let $T,T' \in \mathcal{B}_1(\e,\e')$ with the block matrices \eqref{eq31nou} satisfying the conditions
\begin{equation}\label{eq32nou}
T_0^*T_1+T_2^*T_3=T_0^{'*}T_1'+T_2^{'*}T_3'=0, \quad T_0^*T_1'+T_2^*T_3'=0, \quad T_1^*T_0'+T_3^*T_2'=0.
\end{equation}
If the column matrices $C_0,C_0'$ and $C_1,C_1'$ in \eqref{eq31nou} are Shmul'yan equivalent, respectively, then $T$ and $T'$ are Shmul'yan equivalent.
\end{proposition}

\begin{proof}
Let us assume that the contractions $C_0$ and $C_0'$, respectively $C_1$ and $C_1'$, are Shmul'yan equivalent, where $C_0=\begin{pmatrix} T_0 & T_2 \end{pmatrix}^{{\rm tr}}$, $C'_0=\begin{pmatrix} T'_0 & T'_2 \end{pmatrix}^{{\rm tr}}$, $C_1=\begin{pmatrix} T_1 & T_3 \end{pmatrix}^{{\rm tr}}$, $C_1'=\begin{pmatrix} T_1' & T_3' \end{pmatrix}^{{\rm tr}}$. Then $D_{C_0}^2=I-T_0^*T_0-T_2^*T_2$, $D_{C_1}^2=I-T_1^*T_1-T_3^*T_3$ and $D_{C_0'}^2$, $D_{C_1'}^2$ have similar expressions. Now using \eqref{eq31nou} and the first two conditions from \eqref{eq32nou} we get
$$
D_T=D_{C_0}\oplus D_{C_1}, \quad D_{T'}=D_{C_0'}\oplus D_{C_1'}.
$$

Since $C_0$ and $C_0'$ are Shmul'yan equivalent, by Theorem \ref{thm:1.2} there exists $Y_0 \in \mathcal{B}(\D_{C_0'}, \D_{C_0})$ such that $I-C_0^*C_0'=D_{C_0}Y_0D_{C_0'}$ and similarly, there exists $Y_1\in \mathcal{B}(\D_{C_1'}, \D_{C_1})$ such that $I-C_1^*C_1'=D_{C_1}Y_1D_{C_1'}$. Taking $Y\in \mathcal{B}(\D_T, \D_{T'})$ given by $Y=Y_0\oplus Y_1$, acting from $\D_{T'}=\D_{C_0'}\oplus \D_{C_1'}$ into $\D_T=\D_{C_0} \oplus \D_{C_1}$, and expressing the block matrices below on the same decompositions of $\e$ and $\e'$ as in \eqref{eq31nou}, we get
\begin{eqnarray*}
D_TYD_{T'}&=&
\begin{pmatrix}
D_{C_0}Y_0D_{C_0'} & 0\\
0 & D_{C_1}Y_1D_{C_1'}
\end{pmatrix}
=
\begin{pmatrix}
I-C_0^*C_0' & 0\\
0 & I-C_1^*C_1'
\end{pmatrix}\\
&=&
\begin{pmatrix}
I- T_0^*T_0'-T_2^*T_2' & 0\\
0 & I-T_1^*T_1'-T_3^*T_3'
\end{pmatrix}.
 \end{eqnarray*}
On the other hand, using the last two conditions in \eqref{eq32nou} we obtain
$$
I-T^*T'=
\begin{pmatrix}
I- T_0^*T_0'-T_2^*T_2' & -(T_0^*T_1'+T_2^*T_3')\\
-(T_1^*T_0'+T_3'T_2') & I-T_1^*T_1'-T_3^*T_3'
\end{pmatrix}
=D_TYD_{T'}.
$$
Using Theorem \ref{thm:1.2} we infer that $T$ and $T'$ are Shmul'yan equivalent.\end{proof}

A first application of this proposition refers to the decomposition of an operator $T\in \mathcal{B}_1(\e,\e')$ into a unitary part and a pure part; see \cite[Ch. V]{FF}. More specifically, the operator $T$ can be written as
$$
T=U \oplus Q : \e=\n(D_T)\oplus \mathcal{D}_T \to \e'=\n(D_{T^*})\oplus \mathcal{D}_{T^*},
$$
where $U \in \mathcal{B}(\n(D_T),\n(D_{T^*}))$ is unitary, and $Q\in \n(\mathcal{D}_T,\mathcal{D}_{T^*})$ is a {\it pure contraction}, that is $\|Qh\|<\|h\|$ for all nonzero $h\in \h$. If $T' =U' \oplus Q' \in \mathcal{B}_1 (\e,\e')$ with $U'$ the unitary part and $Q'$ the pure part of $T'$ and with $\D_T=\D_{T'}$, $\D_{T^*}=\D_{T^{'*}}$ (necessary conditions for the Shmul'yan equivalence), then $T$ and $T'$ satisfy trivially the conditions \eqref{eq32nou}. So, by Proposition \ref{pr35nou} the Shmul'yan equivalence of $T$ and $T'$ means $U=U'$ and that $Q$ and $Q'$ are Shmul'yan equivalent in $\mathcal{B}_1(\D_T,\D_{T^*})$. We reobtain in this way \cite[Corollary 2]{Sh}.

In order to state the second application, we recall that $\n(I-S_T)$ is an invariant subspace for each contraction $T\in \mathcal{B}(\h)$, and $T|_{\n(I-S_T)}$ is an isometry.
For another contraction $T'$ which belongs to the Shmul'yan part of $T$ it is necessary that $\n(I-S_T)$ is also invariant (even $\n(I-S_T)=\n(I-S_{T'})$ because $T'$ will be in the Harnack part of $T$) and $T'=T$ on $\n(I-S_T)$; see \cite[Lemma 5.1]{BST}. More precisely we have the following

\begin{corollary}\label{co36}
Let $T,T'\in \mathcal{B}_1(\h)$ be such that $\h_0=\n(I-S_T)$ is invariant for $T'$. For $T,T' : \h_0 \oplus \h_0^{\perp} \mapsto\h_0 \oplus \h_0^{\perp}$, consider the block matrices
\begin{equation}\label{eq33nou}
T=
\begin{pmatrix}
V & R\\
0 & Q
\end{pmatrix},
\quad
T'=
\begin{pmatrix}
V' & R'\\
0 & Q'
\end{pmatrix},
\end{equation}
where $V$ is an isometry and $V^*R=0$. Then $T$ and $T'$ are Shmul'yan equivalent if and only if $V=V'$ and the contractions $\begin{pmatrix} R & Q \end{pmatrix}^{{\rm tr}}$ and $\begin{pmatrix} R' & Q' \end{pmatrix}^{{\rm tr}}$ are Shmul'yan equivalent in $\mathcal{B}_1(\h_0^{\perp},\h)$.
\end{corollary}

\begin{proof}
 The ``if'' assertion follows by Proposition \ref{pr35nou}, because the conditions \eqref{eq32nou} are clearly satisfied when $V=V'$, so $V^*R'=0$. The ``only if'' assertion is immediate because if $T$ and $T'$ are Shmul'yan equivalent they are also Harnack equivalent and so are $V=T|_{\h_0}$ with $V'=T'|_{\h_0}$, hence $V=V'$. Then it is easy to see that the Shmul'yan equivalence of $T$ and $T'$ reduces to that of the column operators $\begin{pmatrix} R & Q \end{pmatrix}^{{\rm tr}}$ and $\begin{pmatrix} R' & Q' \end{pmatrix}^{{\rm tr}}$.
\end{proof}

In the case when $T$ is a quasi-isometry on $\h$ i.e. $T^*T=T^{*2}T^2$ (alternatively, $Q=0$ in \eqref{eq33nou}) we have a more precise assertion.

\begin{proposition}\label{pr37}
A contraction $T'$ on $\h$ is Shmul'yan equivalent to a quasi-isometry $T$ on $\h$ if and only if $\n(I-S_T)$ is invariant for $T'$, $\R(Q')\subset \R(D_R)=\R(D_{T'})$ and $V$, $R$ Shmul'yan dominate $V'$, $R'$, respectively, where $V,R$, $V', R',Q'$ are the operators from the block matrices \eqref{eq33nou} of $T$ and $T'$.
\end{proposition}

\begin{proof}
Assume that $T$ and $T'$ are Shmul'yan equivalent. Then so are $R$ and $R'$ and $\R(D_T)=\R(D_{T'})$. As $Q=0$ in \eqref{eq33nou}, $T$ being quasi-isometry, the previous equality gives $\R(D_R)=\R((I-R^{'*}R'-Q^{'*}Q')^{1/2})=\R(D_{R'})$. So, by Douglas's result in \cite{D} there are the constants $c_0,c_1>1$ such that
$$
D_R^2\le c_0(D_{R'}^2-Q^{'*}Q')\le c_0 (c_1 D_R^2-Q^{'*}Q'),
$$
whence $c_0Q^{'*}Q'\le (c_0c_1-1)D_R^2$. This implies, again by Douglas's result, that $\R(Q')\subset \R(D_R)=\R(D_{T'})$. Hence an implication is proved.

Conversely, suppose that $T,T'$ have the matrix representations \eqref{eq33nou} and  $V,R$ Shmul'yan dominate $V',R'$, respectively, with $\R(Q')\subset \R(D_R)=\R(D_{T'})$. Then
$R' \PrecS R$ means $R-R'=D_{R^*}X_0D_{T'}$, while the inclusion $\R(Q') \subset \R(\D_{T'})$ provides $Q'=X_1D_{T'}$, for some operators $X_0\in \mathcal{B}(\D_{T'},\D_{R^*})$ and $X_1\in \mathcal{B}(\D_{T'},\overline{\R(Q')})$.
Thus one infers
\begin{eqnarray*}
\begin{pmatrix}
R\\
0
\end{pmatrix}
-
\begin{pmatrix}
R' \\
Q'
\end{pmatrix}
&=&
\begin{pmatrix}
D_{R^*}X_0D_{T'}\\
-X_1D_{T'}
\end{pmatrix}\\
&=&
\begin{pmatrix}
D_{R^*} & 0\\
0 & I
\end{pmatrix}
\begin{pmatrix}
X_0\\
-X_1
\end{pmatrix}
(D_{R'}^2-Q^{'*}Q')^{1/2}\\
&=& D_{C^*}XD_{C'},
\end{eqnarray*}
where
$C=\begin{pmatrix}
R & 0
\end{pmatrix}^{{\rm tr}}
$,
$C'=\begin{pmatrix}
R' & Q'
\end{pmatrix}^{{\rm tr}}$
and
$X= \begin{pmatrix} X_0 & -X_1\end{pmatrix}^{{\rm tr}}$.
This means that the column operators $C$ and $C'$ in the matrices from \eqref{eq33nou} are Shmul'yan equivalent. On the other hand, the relation $V' \PrecS V$ forces $V'=V$ because $V$ is an isometry. Thus, by Corollary \ref{co36} we conclude that $T$ and $T'$ are Shmul'yan equivalent.
\end{proof}

\bigskip

{\bf Harnack and Shmul'yan parts of partial isometries.}  Another application refers to partial isometries $T\in \mathcal{B}_1(\h,\h')$. Such operators can be written as the direct sum $T=U\oplus 0$ from $\h=\R(T^*)\oplus \n(T)$ into $\h'=\R(T)\oplus \n(T^*)$. Therefore by Theorem \ref{thm:1.2} any contraction $T'\in \mathcal{B}(\h,\h')$ in the Shmul'yan part of $T$ has the form $T=U\oplus Z$ with $\|Z\|<1$, with respect to the same decompositions of $\h$ and $\h'$. Conversely, any $T'$ of this form is Shmul'yan equivalent to $T$, by Proposition \ref{pr35nou}. Thus, one reobtains the known result from \cite[Theorem 3]{Sh} which completely describes the Shmul'yan part of a partial isometry.

It is known (see \cite[Corollary 3.4]{KSS}) that for isometries and and joints of isometries, called coisometries, the Harnack and Shmul'yan parts coincide and are equal to a singleton. In general, for a partial isometry we have the following

\begin{theorem}\label{pr38}
 The Harnack part $\Delta(T)$ of a partial isometry $T\in\mathcal{B}_1(\h)$ coincides with its Shmul'yan part $\Delta_{Sh}(T)$ and $T$ is the only partial isometry in $\Delta(T)$.
\end{theorem}

\begin{proof}
Let $T' \in \Delta (T)$. Then $T'$ and $T$ are Z-equivalent in the sense of \cite{AST} which implies that they have the same defect spaces. So, in this case we have
$$
\n(D_{T'})=\n(D_T)=\R(T^*), \quad \mathcal{D}_{T'}=\mathcal{D}_T=\n(T),
$$
and similarly when $T',T$ are replaced by $T'^*,T^*$ respectively. Also, we have $T'=T=U$ on $\n(D_T)$. Therefore $T'=U\oplus Z$ acts from $\h=\n(D_T)\oplus \n(T)$ into $\h=\n(D_{T^*})\oplus \n(T^*)$, with $Z$ being the pure part of $T'$. The Harnack domination of $T$ by $T'$ implies (see \cite{AST})
$$
\|T'h\|^2=\|(T-T')h\|^2\le c \|D_{T'}h\|^2 \quad (h\in \n(T))
$$
with some constant $c\ge 1$, that is
$$
\|T'h\|^2 \le \frac{c}{c+1}\|h\|^2 \quad (h\in \n(T)).
$$
From $Z=P_{\n(T^*)}T'|_{\n(T)}$ it follows that $\|Z\|<1$. So, $T=U\oplus 0$ and $T'=U\oplus Z$ are Shmul'yan equivalent as we remarked before, therefore $\Delta_T \subset \Delta _{{\rm Sh}}(T)$. The converse inclusion being obvious we conclude that $\Delta_T=\Delta _{{\rm Sh}}(T)$. It is clear that each $T'\in \Delta _T$, $T'\neq T$ (i.e. $Z \neq 0$) cannot be a partial isometry. The assertion is proved.
\end{proof}

Other characterizations of contractions which are Harnack equivalent to a partial isometry are included in the following result.

\begin{corollary}\label{pr39}
The following statements are equivalent for a contraction $T$ on $\h$:
\begin{itemize}
\item[(i)] $T$ belongs to a Harnack (Shmul'yan) part of a partial isometry;

\item[(ii)] The range of the defect operator $D_T$ is closed;

\item[(iii)] $\{(T^*T)^n\}_{n\in \N}$ uniformly converges in $\B$ to the orthogonal projection on $\n(D_T)$.
\end{itemize}
\end{corollary}

\begin{proof}
Let $W=U\oplus 0$ be a partial isometry. If $T$ is in the same Harnack part as $W$,
then, as we already remarked, one has $T=U \oplus Z$ with $\|Z\|<1$. So $D_T=0\oplus D_Z$ and $D_Z$ is invertible in $\mathcal{B}(\mathcal{D}_T)$. Therefore
$$
(T^*T)^n=I\oplus (Z^*Z)^n\to I \oplus 0=P_{\n(D_T)},
$$
the convergence being uniform in $\B$. Hence (i) implies (iii).

Now (iii) implies that the Ces\`aro means $[I+T^*T + (T^*T)^2 + \cdots (T^*T)^{n}]/(n+1)$ of $T^*T$ also converge uniformly
to $P_{\n(D_T)}$. Using the uniform ergodic theorem of M.~Lin~\cite{L}, this
implies that $D_T\h$ is closed. So (iii) implies (ii).

Finally, suppose that $D_T\h$ is closed. Then, using the decomposition of $T=T_u\oplus T_p$ into unitary and pure parts, we infer that $D_{T_p}\mathcal{D}_T$ is also closed. As $D_{T_p}$ is injective one has $1\notin \sigma (T_p^*T_p)$, which implies $\|T_p\|<1$. Then, by Theorem \ref{pr38}, it follows that $T$ belongs to the Harnack (i.e. Shmul'yan) part of the partial isometry $T_u\oplus 0$. Thus (ii) implies (i).
\end{proof}

\begin{corollary}\label{co310}
Any compact contraction belongs to the Harnack part of a partial isometry.
\end{corollary}

\begin{proof}
Assume that $T$ is compact on $\h$. Then so is $T^*T$ and either $1$ does not belong to its spectrum, or $1$ is an isolated eigenvalue of $T^*T$. In the first case one has $\R(D_T)=\h$. In the other case one has $T^*T=I\oplus E$ on $\h=\n(D_T)\oplus \D_T$ (because $\n(D_T)$ reduces $T^*T$) and $E$ has the spectral radius less than $1$. Hence $\|E^n\|\to 0$, which implies that $\{(T^*T)^n\}$ converges uniformly to the orthogonal projection on $\n(D_T)$. By Corollary \ref{pr39} we conclude that $\Delta(T)$ contains a partial isometry.
\end{proof}

\begin{remark}\label{re36nou}
Corollary \ref{pr39} can be obtained using results from \cite{Sh}. According to \cite{Sh} a contraction $T$ on $\h$ is called {\it regular} if its pure part $Q \in \mathcal{B}(\D_T,\D_{T^*})$ is a strict contraction, \emph{i.e.} $\|Q\|<1$. In other words $T$ is regular if and only if it satisfies the condition (i) of Corollary \ref{pr39}. Also, with this terminology, in \cite{Sh} it was remarked that the assertions (i) and (ii) are equivalent, while Lin's theorem in \cite{L} gives the equivalence of (ii) and (iii).
Also, Corollary \ref{co310} shows that any compact operator is regular.
\end{remark}

Now, one can see, in a more general context, when a contraction $T$ is Harnack equivalent with a partial isometry, by using the matrix representation of $T$ from \eqref{eq33nou}.

The following lemma is probably known, but as we lack a reference we include its short proof.

\begin{lemma}\label{le312}
Let $T$ be a contraction having the representation \eqref{eq33nou} on $\h=\n(I-S_T) \oplus \overline{\R(I-S_T)}$, with $V,R,Q$ as in \eqref{eq33nou}. The following are equivalent :
\begin{itemize}
\item[(i)] $\n(D_T)$ is an invariant subspace for $T$;

\item[(ii)] $\n(D_{T^*}) \subset \n(D_T)$;

\item[(iii)] $Q$ is a pure contraction on $\overline{\R(I-S_T)}$.
\end{itemize}
\end{lemma}

\begin{proof}
Assuming (i) we have $\n(D_T)=\n(I-S_T)$ and $\overline{\R(I-S_T)}=\mathcal{D}_T$. As $T|_{\mathcal{D}_T}$ is a pure contraction from $\mathcal{D}_T$ into $\mathcal{D}_{T^*}$, it follows that $Q=P_{\mathcal{D}_T}T|_{\mathcal{D}_T}$ is pure on $\mathcal{D}_T$. So (i) implies (iii). Now (iii) says that $Q^*=T^*|_{\overline{\R(I-S_T)}}$ is a pure contraction. Therefore $\overline{\R(I-S_T)} \subset \mathcal{D}_{T^*}$, which yields $\n(D_{T^*}) \subset \n(I-S_T) \subset \n(D_T)$. Hence (iii) implies (ii), and obviously (ii) implies (i).
\end{proof}

The conditions of this lemma are satisfied for many contractions, in particular for hyponormal contractions (see \cite{K}). Clearly, if $T$ satisfies (one of) the conditions of Lemma \ref{le312}, then any $T' \in \Delta(T)$ satisfies the same conditions.

When $T$ satisfies one of the equivalent assertions of Lemma \ref{le312} we have $\n(I-S_T)=\n(D_T)$ and the representation \eqref{eq33nou} of $T$ can be related to the decomposition of $T=U\oplus \widetilde{Q}$ into a unitary part $U \in \mathcal{B}(\n(D_T),\n(D_{T^*}))$ and a pure contraction $\widetilde{Q}\in \mathcal{B}(\mathcal{D}_T, \mathcal{D}_{T^*})$. Since $\mathcal{D}_T \subset \mathcal{D}_{T^*}$ in this case, it follows that
$$
V=J_1U, \quad \widetilde{Q}= \begin{pmatrix} R\\ Q \end{pmatrix} : \mathcal{D}_T \to \mathcal{D}_{T^*} =\mathcal{D}_T^{\perp} \oplus \mathcal{D}_T.
$$
Here $J_1$ is the natural embedding of $\n(D_{T^*})$ into $\n(D_T)$. Therefore, using Theorem \ref{pr38} and Lemma \ref{le312}, we obtain the following result.

\begin{corollary}\label{co313}
Let $T$ be a contraction on $\h$ and let $R,Q$ denote the contractions given by
$$
R=P_{\n(I-S_T)}T|_{\overline{\R(I-S_T)}}, \quad Q=P_{\overline{\R(I-S_T)}}T |_{\overline{\R(I-S_T)}}.
$$
The following are equivalent :
\begin{itemize}
\item[(i)] The Harnack part of $T$ contains a partial isometry and $Q$ is a pure contraction;

\item[(ii)] $\|R^*R +Q^*Q\|<1$.
\end{itemize}
\end{corollary}

Clearly, the above condition (ii) is stronger than the combination of $\|R\|<1$ and $\|Q\|<1$. Also, (i) does not imply (ii) without the requirement that $Q$ is pure (that is $\n(D_{T^*}) \subset \n(D_T)$).

We record finally another special case which complements Theorem \ref{te210} and Corollary \ref{pr39}. Recall that $T$ is quasi-normal if $TT^*T=T^*T^2$.

\begin{proposition}\label{pr311}
The following are equivalent for a quasi-normal contraction $T$:
\begin{itemize}
\item[(i)] $T$ belongs to the Harnack part of a partial isometry;

\item[(ii)] The $C_{00}$-part of $T$ is a strict contraction on $\n(S_T)$;

\item[(iii)] An iterate $T^{n_0}$ with $n_0>1$ belongs to the Harnack part of a partial isometry;

\item[(iv)] The sequence $\{T^{*n}T^n\}$ uniformly converges in $\B$ to $S_T=P_{\n(D_T)}$.
\end{itemize}

If $T$ satisfies one of these equivalent assertions, then $T$ belongs to the Harnack part of the operator $V \oplus 0$ acting on $\h= \n(I-S_T) \oplus \n(S_T)$, where $V=T|_{\n(I-S_T)}$.
\end{proposition}

\begin{proof}
The equivalence $(i)\Leftrightarrow (iv)$ is immediate using Corollary \ref{pr39} (iii), since $TT^*T=T^*T^2$, which gives $(T^*T)^n=T^{*n}T^n$ for $n\ge 1$. We also obtain that $S_T=P_{\n(D_T)}$, by Corollary \ref{pr39} (iii) and the fact that $S_T$ is defined as the strong limit of $\{T^{*n}T^n\}$.

We write now $T=V \oplus Q$ with respect to the decomposition $\h=\n(I-S_T) \oplus \n(S_T)$, where $V$ is an isometry and $Q$ is of class $C_{0 \cdot}$. This orthogonal decomposition of $\h$ holds always because $S_T$ like $S_{T^*}$ are orthogonal projections when $T$ is quasi-normal (see for instance \cite{K, LSExtr}). In addition, $\n(I-S_{T^*}) \subset \h$ is the unitary part of $T$. This implies that $\n(S_T) \subset \n(S_{T^*})$. As both subspaces reduce $T$ it follows that $Q$ is of class $C_{00}$. Notice also that $Q$ is a quasi-normal operator. Thus, the convergence (iv) means $\|(Q^*Q)^n\|\to 0$, that is $\sigma(Q^*Q)\subset \De$, or equivalently $\|Q\|<1$. Hence the assertions (ii) and (iv) are equivalent.

Condition (ii) implies (iii) for every integer $n_0>1$. Conversely, if (iii) holds for some $n_0>1$ then by applying the above equivalence (i) $\Leftrightarrow$ (ii) to $T^{n_0}$ one has $\|Q^{n_0}\|<1$. This implies $\sigma (Q) \subset \De$, and finally $\|Q\|<1$. Consequently (ii) and (iii) are equivalent.
We have proved that all assertions (i)-(iv) are equivalent. In addition, using Proposition \ref{pr35nou}, we remark that $T=V\oplus Q$ is Shmul'yan (so Harnack) equivalent to $W=V\oplus 0$ when $\|Q\|<1$ (as in (ii)). Thus, in this case, $T$ belongs to the Harnack part of a partial (quasi-normal) isometry $W$. This ends the proof.
\end{proof}

Concerning the statement (iv) in Proposition \ref{pr311} we know that the sequence $\{T^{*n}T^n\}$ always converges strongly to $S_T$, so the meaning of Proposition \ref{pr311}, (iv), is that the convergence is even uniform whenever $T$ belongs to the Harncack part of a partial isometry.

\medskip

\section{Connections with the Shmul'yan-ter Horst domination of Schur class functions}\label{sec:4}
\medskip

The next topic is the Shmul'yan-ter Horst domination by partial isometries. The first result is the following

\begin{theorem}\label{te44}
 A partial isometry $T$ acting between two separable Hilbert spaces $\mathcal{E}$ and $\e'$, viewed as a constant function in $H_1^{\infty} (\mathcal{E})$, Shmul'yan-ter Horst dominates any function $G \in H_1^{\infty} (\mathcal{E}, \e')$ with $G(\De) \subset \Delta_{{\rm Sh}} (T)$.
\end{theorem}

\begin{proof}
Theorem \ref{te31nou} for the partial isometry $T$ and $T'\in \Delta_{{\rm Sh}}(T)$ implies the existence of $F \in H_1^{\infty}(\mathcal{E},\e')$, $F(\De) \subset \Delta_{{\rm Sh}}(T)$, with $F(0)=T$ and $F(z_0)=T'$ for some $z_0\in \De$. In addition, the function $F$ is given by the formula
$$
F(\lambda )=T+ \lambda D_{T^*}F_1(\lambda )[I+ \lambda T^* F_1(\lambda )]^{-1}D_T, \quad \lambda \in \De,
$$
where $F_1\in H_1^{\infty}(\D_T, \D_{T^*})$ (see \cite[Ch. XIII, 3, Example 3.7]{FF}). Since $T$ is a partial isometry one has $\D_{T^*}=\n(T^*)$, so $T^*F_1(\lambda )=0$ for $\lambda \in \De$. Therefore
\begin{equation}\label{ec42}
F(\lambda )=T+\lambda D_{T^*}F_1(\lambda )D_T, \quad \lambda \in \De,
\end{equation}
with $\sup_{\lambda \in \De} \|\lambda F_1(\lambda )\|\le 1$. Using \cite[Theorem 2.6]{H},
this implies that $F$ is dominated by the constant function $T$ in $H_1^{\infty} (\mathcal{E},\e')$.

Let now $G\in H_1^{\infty} (\mathcal{E},\e')$ be an arbitrary function with its range $G(\De)$ contained in $\Delta_{{\rm Sh}} (T)$. Let $\lambda \in \De$. Then by the previous conclusion applied to the  contractions $T$ and $G(\lambda )$ there exist $z_{\lambda }\in \De$ and $K_{\lambda }\in H_1^{\infty} (\mathcal{E},\e')$ such that $K_{\lambda }(0)=T$ and $K_{\lambda} (z_{\lambda})=G(\lambda )$. As above, there is a function $F_{\lambda} \in H_1^{\infty} (\D_T, \D_{T^*})$ with
$$
K_{\lambda }(z) =T+ D_{T^*}F_{\lambda }(z)D_T, \quad z\in \De.
$$
For $z=z_{\lambda }$ this yields
\begin{equation}\label{ec43}
G(\lambda )=T+D_{T^*}F_{\lambda }(z_{\lambda })D_T=T+D_{T^*}Q_{\lambda }D_T, \quad \lambda \in \De .
\end{equation}
Here $Q_{\lambda }:=F_{\lambda }(z_{\lambda })\in \mathcal{B}(\D_T, \D_{T^*})$ satisfy $\|Q_{\lambda }\|\le 1$ for $\lambda \in \De$.

Notice that the contraction $Q_{\lambda }$ is uniquely determined by $\lambda $, $G(\lambda )$, $z_{\lambda }$ and $K_{\lambda }\in H_1^{\infty} (\mathcal{E},\e')$ as above. Indeed, assume that for $\lambda \in \De$ there are $z_{\lambda }' \in \De$ and $K_{\lambda }' \in H_1^{\infty}(\e,\e')$ such that $K_{\lambda }'(0)=T$ and $K_{\lambda }'(z_{\lambda }')=G(\lambda )$. Then $K_{\lambda }'$ has the form of $F$ in \eqref{ec42}, hence there exists $F_{\lambda }'\in H_1^{\infty}(\D_T, \D_{T^*})$ such that
$$
K_{\lambda }'(z)=T+D_{T^*}F_{\lambda }'(z)D_T, \quad z\in \De.
$$
This implies, for $z=z_{\lambda }'$, the identity
$$
G(\lambda )=T+D_{T^*}F_{\lambda }'(z_{\lambda }')D_T, \quad \lambda \in \De.
$$
From this expression and that in \eqref{ec43} for $G$ we infer
$$
D_{T^*}(Q_{\lambda}-F_{\lambda }'(z_{\lambda }'))D_T=0, \quad \lambda \in \De,
$$
that is $F_{\lambda }'(z_{\lambda }')=Q_{\lambda }=F_{\lambda }(z_{\lambda })$, $\lambda \in \De$. Hence the mapping $\widetilde{Q}$ from $\De$ into $\mathcal{B} (\D_T, \D_{T^*})$ defined by the formula
$$
\widetilde{Q}(\lambda )=Q_{\lambda }, \quad \lambda \in \De,
$$
is well-defined and bounded on $\De$. Then by \cite[Theorem 0.2]{H} it follows that the function $G$ satisfying the relation \eqref{ec43} is Shmul'yan-ter Horst dominated by the constant function $T$ in $H_1^{\infty}(\mathcal{E},\e')$. This ends the proof.
\end{proof}

\begin{remark}\label{rm:pi}
\rm
Concerning the assertion $G \PrecH T$ of Theorem \ref{te44}, we have the following example. Let $\mathcal{E}=\C$ and let $T = 0_{\C}$ be the null operator whose Harnack part is, according to \cite{F}, the set of all strict contractions. The identity function $u(\lambda )=\lambda $ ($\lambda \in \De$) satisfies $u(\De)\subset \De$ and $\|u\|_{\infty}=1$, but only functions $f\in H^{\infty}$ with $\|f\|_{\infty}<1$ Shmul'yan-ter Horst dominate $0_{\C}$.
Hence the functions $u$ and $0_{\C}$ are not Shmul'yan-ter Horst equivalent. This shows that the implication given by the assertion of Theorem \ref{te44} cannot be upgraded to an equivalence.
\end{remark}

Consider now the equivalence class $\Delta _{\infty}(T)$ in the sense of Shmul'yan-ter Horst of a partial isometry $T$, viewed as a constant function in $H_1^{\infty}(\mathcal{E},\e')$. The following result provides a description of $\Delta _{\infty}(T)$ which is similar to that of $\Delta_{{\rm Sh}}(T)=\Delta(T)$.

\begin{theorem}\label{te45}
If $T$ is a partial isometry acting between two separable Hilbert spaces $\mathcal{E}$ and $\e'$, then
\begin{equation}\label{ec44}
\Delta _{\infty}(T) =\{F(\cdot)=T+D_{T^*}F_0(\cdot)D_T\in H_1^{\infty} (\mathcal{E},\e'):F_0 \in H_1^{\infty} (\D_T, \D_{T^*}), \hspace*{1mm} \|F_0\|_{\infty} <1\}.
\end{equation}
\end{theorem}

\begin{proof}
Let $F\in H_1^{\infty}(\mathcal{E},\e') \cap \Delta _{\infty} (T)$, $T$ being a partial isometry from $\mathcal{E}$ into $\e'$. Then by \cite[Theorem 2.6]{H} there exists a norm bounded function $F_0(\cdot)$ on $\De$ with values in $\mathcal{B}(\D_T,\D_{T^*})$ such that
$$
F(\lambda )=T+D_{T^*}F_0(\lambda )D_T, \quad \lambda \in \De.
$$
Since $T$ is a partial isometry (so $\D_{T^*} =\n(T^*)$), we have
$$
F_0(\lambda )=P_*F(\lambda )|_{\D_T}, \quad \lambda \in \De,
$$
where $P_*$ is the projection of $\mathcal{E}'$ onto $\D_{T^*}$. In particular, $F_0 \in H_1^{\infty} (\D_T, \D_{T^*})$.

Now, identifying canonically $\D_T \sim \{0\}\oplus \D_T$ and $\D_{T^*}\sim \{0\} \oplus \D_{T^*}$ as subspaces of $\mathcal{E} =\n(D_T) \oplus \D_T$ and $\mathcal{E}' =\n(D_{T^*})\oplus \D_{T^*}$ respectively, one can consider the function $F_1 \in H_1^{\infty}(\mathcal{E},\e')$ having the representation
\begin{equation}\label{eq:f1}
F_1(\lambda )=0\oplus F_0(\lambda )
\end{equation}
from $\n(D_T) \oplus \D_T$ into $\n(D_{T^*})\oplus \D_{T^*}$, for $\lambda \in \De$. One has $F_1 \PrecH 0$ in $H_1^{\infty}(\mathcal{E},\e')$ (see \cite[Theorem 2.6, or Corollary 2.7 (ii)]{H}).

On the other hand, as $T \PrecH F$ in $H_1^{\infty}(\mathcal{E},\e')$ we have, as above,
$$
T=F(\lambda )+D_{F(\lambda )^*}F_0'(\lambda) D_{F(\lambda )} \quad (\lambda \in \De)
$$
for some $\mathcal{B}(\D_T,\D_{T^*})$-valued norm bounded function $F_0'$ on $\De$. We have used here that $\D_{F(\lambda )}=\D_T$ and $\D_{F(\lambda )^*}=\D_{T^*}$. Indeed, by \cite[Corollary 2.3]{H} we have $F(\lambda) \in \Delta_{{\rm Sh}}(T)$ and $F(\lambda)^* \in \Delta_{{\rm Sh}}(T^*)$, for any $\lambda \in \De$. Therefore, from the two above relations we get
\begin{eqnarray*}
F_0(\lambda ) &=& P_*(F(\lambda )-T)|_{\D_T}=- P_* D_{F(\lambda )^*} F_0'(\lambda )D_{F(\lambda )}|_{\D_T}\\
&=& -D_{F_0(\lambda )^*} F_0'(\lambda ) D_{F_0(\lambda )},
\end{eqnarray*}
for $\lambda \in \De$. We obtain
$0=F_0(\lambda )+D_{F_0(\lambda )^*}F_0'(\lambda )D_{F_0(\lambda )}$
for $\lambda \in \De$. By the above canonical identifications of $\D_T$ and $\D_{T^*}$ this relation yields
$$
0=F_1(\lambda )+D_{F_1(\lambda )^*}F_0'(\lambda )D_{F_1(\lambda )}, \quad \lambda \in \De.
$$
Thus, using \cite[Theorem 2.6]{H}, one has $0 \PrecH F_1$ in $H_1^{\infty} (\mathcal{E},\e')$ and we finally obtain that $F_1$ is equivalent in the Shmul'yan-ter Horst sense with the null function in $H_1^{\infty} (\mathcal{E},\e')$. Then, using \cite[Corollary 2.7 (ii)]{H}, it follows that $\|F_1\|_{\infty} <1$, that is $\|F_0 \|_{\infty} <1$. Thus, a first inclusion between the two sets of \eqref{ec44} has been proved.

The proof of the converse inclusion in \eqref{ec44} is simpler. Indeed, let $F\in H_1^{\infty} (\mathcal{E},\e')$ of the form $F(\cdot)=T+D_{T^*}F_0(\cdot)D_T$, with a function $F_0 \in H_1^{\infty} (\D_T, \D_{T^*})$ such that $\|F_0 \|_{\infty} <1$. Then also $\|F_1\|_{\infty} <1$, $F_1$
being defined as in \eqref{eq:f1}. Using the same result in \cite{H} we have that $F_1$ is in the equivalence class of the null function in $H_1^{\infty} (\mathcal{E},\e')$. Since $F(\lambda)|_{\n(D_T)}=T|_{\n(D_T)}$, $\lambda \in \De$, the function $F$ belongs to the equivalence class of $T$ in $
H_1^{\infty}(\mathcal{E},\e')$. This ends the proof.
\end{proof}



\end{document}